\documentclass{article}
\usepackage[utf8]{inputenc}

\usepackage{style}

\let\origdoublepage\cleardoublepage
\newcommand{\clearemptydoublepage}{%
  \clearpage{\pagestyle{empty}\origdoublepage}}
\let\cleardoublepage\clearemptydoublepage

\title{Optimal Bounds for Tyler's M-Estimator for Elliptical Distributions}
\author{Lap Chi Lau\footnote{lapchi@uwaterloo.ca, University of Waterloo}, Akshay Ramachandran\footnote{aramach@cs.ubc.ca, University of British Columbia}}
\date{May 2025}

\begin{document}

\maketitle

\begin{abstract}%
A fundamental problem in statistics is estimating the shape matrix of an Elliptical distribution. 
This generalizes the familiar problem of Gaussian covariance estimation, for which the sample covariance achieves optimal estimation error.
For Elliptical distributions, Tyler proposed a natural M-estimator and showed strong statistical properties in the asymptotic regime, independent of the underlying distribution.
Numerical experiments show that this estimator performs very well, and that Tyler's iterative procedure converges quickly to the estimator. 
Franks and Moitra recently provided the first distribution-free error bounds in the finite sample setting, as well as the first rigorous convergence analysis of Tyler's iterative procedure. 
However, their results exceed the sample complexity of the Gaussian setting by a $\log^{2} d$ factor. 
We close this gap by proving optimal sample threshold and error bounds for Tyler’s M-estimator for all Elliptical distributions, fully matching the Gaussian result. Moreover, we recover the algorithmic convergence even at this lower sample threshold.
Our approach builds on the operator scaling connection of Franks and Moitra by introducing a novel `pseudorandom' condition, which we call $\infty$-expansion. We show that Elliptical distributions satisfy $\infty$-expansion at the optimal sample threshold, and then prove a novel scaling result for inputs satisfying this condition.
\end{abstract}

\newpage

\section{Introduction} \label{s:intro}





    The covariance matrix of random variable is a natural and useful statistic of high dimensional distributions, as it gives insight into the geometry of the data.
    Estimation of the covariance is therefore a fundamental task in data analysis. 
    For sufficiently nice distributions, such as multivariate Gaussians, the sample covariance is a very accurate estimator that is easy to compute. 
    However in many practical situations, the underlying distribution is less well-behaved, so the sample covariance is less accurate. In fact, for sufficiently heavy-tailed distributions such as the multivariate $t$-distribution, the covariance matrix may not even exist. 

    Elliptical distributions are a well-studied model class of heavy tailed distributions (see \cite{kelker1970distribution}, \cite{gupta2013elliptically}), and include important special cases such as multivariate Gaussian and $t$-distributions. 
    While Elliptical distributions do not always have a well-defined covariance matrix, they are parameterized by a `shape matrix' which captures similar geometric properties. 
    Tyler \cite{tyler} proposed a natural estimator for the shape matrix of elliptical distributions, along with a simple iterative procedure to compute the estimator.
    He was able to prove strong asymptotic guarantees for this estimator, essentially recovering some of the desirable statistical properties of the Gaussian setting.
    Importantly, the asymptotic guarantees shown by Tyler \cite{tyler} are \emph{distribution-free}, in that they are independent of the underlying distribution or shape matrix.
    Both the estimator and the heuristic procedure have been shown to perform well in numerical experiments. 

    Following this seminal result, there have been attempts to show estimation guarantees for Elliptical distributions in the finite sample regime. 
    Soloveychik and Wiesel \cite{soloveychik2014performance} showed that Tyler's M-Estimator achieved optimal error in the Frobenius norm, but with an additional factor that depends on the condition number of the shape matrix.
    Regularized estimators have also been proposed, which can be computed efficiently but do not have the same provable statistical properties as Tyler's estimator. 
    For a thorough discussion of these results, see the survey of Wiesel and Zhang \cite{wiesel2015structured}. 
    
    The first distribution-free guarantees for Tyler's M-Estimator were proven recently by Franks and Moitra in \cite{FM20}; 
    they also gave the first rigorous analysis of Tyler's iterative procedure, showing linear convergence to the estimator.
    These results are nearly optimal, but have sample threshold and error results that exceed the Gaussian setting by $\log d$ factors. 
    It is natural to ask whether Tyler's estimator can be shown to match this optimal guarantee, or whether shape estimation for Elliptical distributions is strictly more difficult than Gaussian covariance estimation.


    In this work, we show optimal sample complexity and error guarantees for Tyler's M-Estimator for the shape matrix of Elliptical distributions. 
    These results are tight, as they match the known lower bounds for the special case of Gaussians. 
    We also recover the algorithmic analysis of \cite{FM20} with fewer samples, showing the same linear convergence of Tyler's iterative procedure at the optimal sample threshold.

    \subsection{Our Results}
    
    Our formal sample complexity result is as follows:

\begin{theorem} \label{t:mainSampleComplexity}
    Given $n \gtrsim \frac{d}{\eps^{2}}$ samples from an elliptical distribution with shape $\Sigma \in \PD(d)$, where $\eps$ is at most a small constant, Tyler's M-Estimator $\hat{\Sigma}$ satisfies 
    \[ \| I_{d} - \Sigma^{1/2} \hat{\Sigma}^{-1} \Sigma^{1/2}\|_{\op} \leq \eps  \]
    with probability $\geq 1 - \exp( - \Omega(\eps^{2} n) )$. 
\end{theorem}

    This improves on the main Theorem 1.1 of \cite{FM20} by removing $\log d$ factors from the sample threshold and error bound. 
    Furthermore, both the sample threshold and error rate are optimal up to constant factors, as they match the lower bound for covariance estimation for the special case of multivariate Gaussian distributions. 
    Theorem 1.2 of \cite{FM20} gives an optimal error guarantee in the Frobenius norm, but only in the large sample setting $n \gtrsim d^{2}$.
    As $\|\cdot\|_{F} \leq \sqrt{d} \|\cdot\|_{\op}$, our \cref{t:mainSampleComplexity} recovers this optimal guarantee while improving the sample requirement by a $d$ factor.  

    Our measure of error in \cref{t:mainSampleComplexity} is known as the relative operator norm. 
    This is the relevant measure for statistical applications, as Arbas et al \cite{arbas2023polynomial} show strong bounds on the relative operator norm bounds between $\Sigma, \hat{\Sigma}$ imply similar strong bounds on KL-divergence and total variation distance between the corresponding Gaussian distributions $N(0,\Sigma), N(0,\hat{\Sigma})$.

    For our second main result, we recover the algorithmic convergence of \cite{FM20} with fewer samples. 

\begin{theorem} \label{t:TylerIterativeAnalysis}
    Given $n \gtrsim d$ samples from an elliptical distribution with shape $\Sigma \in \PD(d)$, with probability $\geq 1 - \exp(-\Omega(n))$ the $T$-th iterate $\overline{\Sigma}_{(T)}$ of Tyler's procedure approximates the M-Estimator $\hat{\Sigma}$ as 
    \[ \| I_{d} - \hat{\Sigma}^{1/2} \overline{\Sigma}_{(T)}^{-1} \hat{\Sigma}^{1/2}\|_{F} \leq \delta  \]
    in $T \lesssim |\log \det \Sigma | + d + \log(1/\delta)$ iterations.  
\end{theorem}

    This improves on the sample requirement of Theorem 1.3 in \cite{FM20} by $\log^{2} d$. Further, this requirement is optimal, as the estimator is not even well-defined for $n < d$.

    In the following subsection, we discuss our techniques, showing how they build on and improve the previous approach of \cite{FM20}. 

    \subsection{Techniques}

    An important observation of Franks and Moitra \cite{FM20} was to show a connection between Tyler's M-Estimator and \emph{frame and operator scaling}. This is an optimization problem over matrices arising in the context of geometric invariant theory, and has recently attracted much interest due to its connections to problems in algebraic complexity (see \cite{GGOW15}, \cite{Towards}). 
    The key technical contribution in \cite{FM20} is to show that, for sufficiently large sample size $n \gtrsim d \log d$, the data from Elliptical distributions satisfy a `quantum expansion' property. 
    They can then appeal to sophisticated scaling results of Kwok et al \cite{KLR} for quantum expanders to prove their results for Tyler's M-Estimator. 

    In this work, we follow a similar approach, improving both parts of the argument to prove optimal guarantees.
    We first identify a stronger `pseudorandom' condition, $\infty$-expansion (see \cref{d:frameInftyExpansion}), and prove that $n \gtrsim d$ samples from an Elliptical distribution satisfies this condition with high probability.
    Then we give a novel analysis of frame scaling, showing stronger operator norm bounds than the results of \cite{KLR} when the input satisfies $\infty$-expansion.
    We believe this result is of independent interest, and in \cref{s:conclusion} we discuss potential future applications to the Paulsen problem in frame theory and the tensor normal model in statistics. 
    Our sample complexity results follow by combining the above two steps. 

    Our proof of the algorithmic result in 
    \cref{t:TylerIterativeAnalysis}
    also crucially uses the connection to scaling. In \cite{FM20}, the authors study 
    a \emph{geodesically convex} optimization formulation for Tyler's M-Estimator. They use this perspective to show: (1) Tyler's iterative procedure can be seen as a natural descent method;
    (2) quantum expansion is related to a geodesic version of \emph{strong convexity}. 
    The convergence follows by standard convex arguments applied to this geodesic setting. 

    In our work, we use a different `pseudorandom' condition, $\infty$-expansion, instead of quantum expansion.
    Nevertheless we show that the algorithmic convergence analysis of \cite{FM20} still follows from our results. 
    Concretely, in \cref{app:InftyImpliesQuantumExpansion} we show that our $\infty$-expansion condition implies quantum expansion, which allows us to apply the same argument as in \cite{FM20} to prove fast convergence of Tyler's iterative procedure at the optimal sample threshold.

\section{Preliminaries and Technical Overview} \label{s:prelim}

In this section we formally define the statistical estimation problem considered in this work.
We also formally define frame scaling and show its connection to Tyler's M-Estimator. 
We end this section with a proof outline of our results, including the main technical ingredients.

\textbf{Notation}: we use $f \lesssim g$ and $f \leq O(g)$ interchangeably to mean that there is a universal constant $C > 0$ such that $f \leq C \cdot g$, and we use $f \gtrsim g$ and $f \geq \Omega(g)$ for the opposite inequality. 
$S^{d-1} \subseteq \R^{d}$ is the set of unit vectors;
$\PD(d)$ is the set of positive definite matrices in $\R^{d \times d}$;
and $\diag(n)$ is the set of diagonal matrices in $\R^{n \times n}$.
For vector $y \in \R^{n}$, $\diag(y)$ is the diagonal matrix with entries $y_{j}$; and by abuse of notation, for matrix $F \in \R^{n \times n}$, $\diag(F)$ is the diagonal matrix with the same entries as $F$ on the diagonal and remaining entries zero.  

\subsection{Elliptical Distributions and Tyler's M-Estimator}

We begin with the formal definition of our statistical model.

  \begin{definition} [Elliptical Distribution] \label{d:EllipticalDist}
      For fixed `shape' matrix $\Sigma \in \PD(d)$ and scalar random variable $u \in \R$, the Elliptical random variable $X \sim E(\Sigma,u)$ is distributed as 
      \[ X := \Sigma^{1/2} V \cdot u   \]
      where $V \sim S^{d-1}$ is a uniformly random unit vector, and $u$ is independent of $V$. 
  \end{definition}

  These generalize the family of Gaussian distributions: for standard normal $g \sim N(0,I_{d})$, the norm $\|g\|_{2}$ is independent of the direction $g/\|g\|_{2}$ which is uniformly distributed on the sphere $S^{d-1}$; therefore $N(0,\Sigma)$ is equivalent to $E(\Sigma, u)$ for $u := \|g\|_{2}$. 
  By computing the second moment, we see that if the covariance matrix of an Elliptical distribution exists, then it must be proportional to the shape matrix. 

  In this work, we study the following estimator for the shape matrix $\Sigma$.

  \begin{definition} [Tyler's M-Estimator] \label{d:TylerMEstimator}
      Given $x_{1}, ..., x_{n} \in \R^{d}$, consider the following equations:
      \[ \frac{d}{n} \sum_{j=1}^{n} \frac{x_{j} x_{j}^{T}}{ x_{j}^{T} \hat{\Sigma}^{-1} x_{j} } = \hat{\Sigma}, \qquad \text{ and } \quad \tr[\hat{\Sigma}] = d.      \]
      If the above equations have a \emph{unique} solution in $\hat{\Sigma} \in \PD(d)$, then Tyler's M-estimator is defined to be that solution; otherwise it is not well-defined.  
  \end{definition}



  

  \subsection{Frame Scaling} \label{ss:PrelimFrameScaling}

  In this subsection, we define frame scaling and describe the connection to Tyler's M-Estimator. 
  Frames are linear algebraic primitives with applications to a variety of fields including coding theory \cite{Frames}, learning theory \cite{diakonikolas2021forster}, and communication complexity \cite{Forster}.
  Formally, they are spanning sets of vectors $V := \{v_{1}, ..., v_{n}\} \in \R^{d \times n}$, and can be thought of as overcomplete bases: $x \in \R^{d}$ can be encoded in a redundant manner in terms of its `frame coefficients' $\{\langle x, v_{j} \rangle\}_{j \in [n]}$. The two most basic properties of frames are the isotropy condition, $V V^{T} = I_{d}$ and the equal-norm condition $\|v_{j}\|_{2}^{2} = \frac{d}{n}$. The first allows for easy reconstruction from frame coefficients $x = \sum_{j \in [n]} \langle x, v_{j} \rangle v_{j}$, and the second gives a balanced encoding where no coefficient is too important on average.  
  The following quantities are used to measure the quality of a frame:

    \begin{definition} \label{d:sizeRowCol}
    Given frame $V \in \R^{d \times n}$, its size is $s(V) := \|V\|_{F}^{2}$, and its error is
    \[ E(V) := d \cdot V V^{T} - s(V) I_{d}, \qquad 
    F(V) := \diag(n \cdot V^{T} V - s(V) I_{n}) . \]
    \end{definition}

    Observe that $E(V)$ measures distance from the isotropy condition, and $F(V)$ measures distance from the equal-norm condition.
    The goal of frame scaling is to transform a given frame to satisfy these two conditions simultaneously. 
    We will use the following measures of error: 

    \begin{definition} \label{d:DeltaOpError}
    For frame $V \in \R^{d \times n}$, the $\ell_{2}$ and $\ell_{\infty}$ error measures are 
    \[ \Delta(V) := \frac{1}{d} \|E(V)\|_{F}^{2} + \frac{1}{n} \|F(V)\|_{F}^{2} , \qquad \|(E, F) \|_{\op} := \max\{ \|E\|_{\op}, \|F\|_{\op} \}  .  \]
    $V$ is $\eps$-doubly balanced if $\|(E, F) \|_{\op} \leq s(V) \cdot \eps$, and is doubly balanced if $\eps=0$. 
    \end{definition}

    We note the following simple relation for later use.  

    \begin{fact} [Lemma 2.15 in \cite{KLR}] \label{f:matrixdbSmallGrad}
    For frame $V \in \R^{d \times n}$, $\Delta(V) \leq \|E(V)\|_{\op}^{2} + \|F(V)\|_{\op}^{2}$. In particular if $V$ is $\eps$-doubly balanced then $\Delta(V) \leq 2 s(V)^{2} \cdot \eps^{2}$. 
    \end{fact}

    We can now formally define frame scaling. 
  
    \begin{definition} [Frame Scaling Problem] \label{d:frameScalingProblem}
        Given frame $U \in \R^{d \times n}$, find left/right scalings $L \in \R^{d \times d}$ and $R \in \diag(n)$ such that $V := L U R$ is doubly balanced:
        \[ V V^{T} = \frac{s(V)}{d} I_{d}, \qquad \forall j \in [n]: \|v_{j}\|_{2}^{2} = \frac{s(V)}{n} .   \]
    \end{definition}


    The key insight in \cite{FM20} is the following connection between frame scaling and Tyler's M-Estimator, which can be verified directly by comparing definitions. 

    \begin{lemma} [Example 2.3 in \cite{FM20}] \label{l:EllipticalFrameCorrespondence}
        Consider input $X = \{x_{1}, ..., x_{n} \} \in \R^{d \times n}$. 
        \begin{itemize}
            \item If $\hat{\Sigma}$ is the M-Estimator for input $\{x_{1}, ..., x_{n}\}$ according to \cref{d:TylerMEstimator}, then the following defines a scaling solution for frame $X$ according to \cref{d:frameScalingProblem}:
            \[ L := \hat{\Sigma}^{-1/2} \quad , \quad R_{jj} := ( \langle x_{j}, \hat{\Sigma}^{-1} x_{j} \rangle)^{-1/2} . \] 
            \item Conversely, if $(L,R)$ is a frame scaling solution for $X$, then the following satisfies the equations in \cref{d:TylerMEstimator} for Tyler's M-Estimator:
            \[  \hat{\Sigma} = \frac{d \cdot ( L^{T} L )^{-1}}{\tr[ (L^{T} L)^{-1} ]}  .  \]
        \end{itemize}
    \end{lemma}

    Therefore, analyzing Tyler's M-Estimator for input $\{x_{1}, ..., x_{n}\}$ is equivalent to analyzing the scaling solution for frame $X$.

    \subsection{Technical Overview} \label{ss:techOverview}

    In this subsection, we give an outline of our proof, including our main technical contributions. 

    The first step, following \cite{FM20}, is to notice some useful invariance properties: the equations in \cref{d:TylerMEstimator} are unaffected by scalar transformation $x_{j} \to \lambda x_{j}$ for $\lambda \in \R$; similarly, a direct computation shows that if $\hat{\Sigma}$ satisfies the equations in \cref{d:TylerMEstimator} for input $\{x_{j}\}$, then for any invertible $A \in \R^{d \times d}$, $A \hat{\Sigma} A^{T}$ satisfies the equations for transformed input $\{ A x_{j}\}$; finally, the relative operator error considered in \cref{t:mainSampleComplexity} is invariant under linear transformations: 
    \[ (A \Sigma A^{T})^{1/2} (A \hat{\Sigma} A^{T})^{-1} (A \Sigma A^{T})^{1/2} = \Sigma^{1/2} \hat{\Sigma}^{-1} \Sigma^{1/2}  .   \]
    Combining these properties give a reduction from general Elliptical distributions to a simpler special case:


        \begin{lemma} [Observation 3.1 in \cite{FM20}] \label{l:EllipticalReductiontoI}
        Consider samples $x_{1}, ..., x_{n} \sim E(\Sigma, u)$ from the elliptical distribution with shape matrix $\Sigma$. 
        Then the `normalized' data $v_{j} := \frac{\Sigma^{-1/2} v_{j}}{\|\Sigma^{-1/2} v_{j}\|_{2}}$ follows Elliptical distribution $E(I_{d}, 1)$. Further, if $\hat{\Sigma}_{X}, \hat{\Sigma}_{V}$ are the M-estimators for $X,V$ according to \cref{d:TylerMEstimator}, then  
    \[ \hat{\Sigma}_{V} = \Sigma^{-1/2} \hat{\Sigma}_{X} \Sigma^{-1/2} \quad \text{and} \qquad 
    \| I_{d} - \Sigma^{1/2} \hat{\Sigma}_{X}^{-1} \Sigma^{1/2} \|_{\op} = \| I_{d} - \hat{\Sigma}_{V}^{-1} \|_{\op}  .  \]
    \end{lemma}

    This shows, in order to prove error bounds for Tyler's M-Estimator in the relative operator norm, it suffices to just consider the special case of $E(I_{d}, 1)$. 
    We note that this reduction is only for the sake of analysis, as we do not have knowledge of the true shape matrix when computing the estimator.
    This is equivalent to the uniform distribution on the sphere $S^{d-1}$, and we will use concentration properties of this simple distribution for our results. 


    Next, we use the connection between Tyler's M-Estimator and frame scaling as described in
    \cref{l:EllipticalFrameCorrespondence}.
    If we can prove strong bounds on the left scaling solution for frame $V \in \R^{d \times n}$ with columns $v_{j} \sim S^{d-1}$, then this transfers directly to error bounds on the estimator. 

    In \cite{FM20}, the authors used the following property to analyze frame scaling. 

    \begin{definition} [Quantum Expansion] \label{d:frameQuantumExpansion}
    Frame $V \in \R^{d \times n}$ is a $(1-\lambda)$ quantum expander if 
        \[ \forall y \perp 1_{n}, \|y\|_{2} \leq 1 : \quad \bigg\| \sum_{j \in [n]} y_{j} v_{j} v_{j}^{*} \bigg\|_{F} \leq \frac{s(V) (1 - \lambda)}{\sqrt{dn}}  .    \]
    \end{definition}

    This is also called the `spectral gap condition' in the work of Kwok et al \cite{KLR}, where the authors show the following strong scaling bound.

    \begin{theorem} [Theorem 1.7 in \cite{KLR}] \label{t:KLR}
        Let frame $V \in \R^{d \times n}$ be $\eps$-doubly balanced and satisfy $(1-\lambda)$-quantum expansion. If $\lambda^{2} \gtrsim \eps \log d$ then the scaling solution satisfies 
        \[ \|L - I_{d} \|_{\op} \lesssim \frac{\eps \log d}{\lambda} .    \]
    \end{theorem}

    The main technical contribution of \cite{FM20} is to show that random instances $v_{1}, ..., v_{n} \sim S^{d-1}$ are quantum expanders with high probability when $n \gtrsim d \log d$. This allows them to apply the above sophisticated scaling result of \cite{KLR} to show near-optimal bounds for Tyler's M-Estimator. 
    
    Our key insight for proving optimal bounds involves the following novel notion of expansion. 

    \begin{definition} [$\infty$-Expansion] \label{d:frameInftyExpansion}
    Frame $V \in \R^{d \times n}$ of size $s(V)$ satisfies $(1-\lambda)$-$\infty$-expansion if 
        \[ \forall y \perp 1_{n}, \|y\|_{\infty} \leq 1 : \quad \bigg\| \sum_{j \in [n]} y_{j} v_{j} v_{j}^{*} \bigg\|_{\op} \leq \frac{s(V) (1 - \lambda)}{d}  .    \]
    \end{definition}

    As compared to quantum expansion, \cref{d:frameQuantumExpansion}, this condition involves all $\|y\|_{\infty} \leq 1$ as opposed to all $\|y\|_{2} \leq 1$, while the output is now bounded in terms of the operator norm instead of the Frobenius norm. This is in fact a stronger property, as in \cref{t:InftyImpliesQ} we show $\infty$-expansion implies quantum expansion. 
    We use this condition to give a new improved analysis of frame scaling.

    \begin{theorem} \label{t:frameMainInfty}
        Let frame $V \in \R^{d \times n}$ be $\eps$-doubly balanced and satisfy $(1-\lambda)$-$\infty$-expansion. If $\lambda^{2} \gtrsim \eps$ then the scaling solution satisfies 
        \[ \|L - I_{d} \|_{\op} \lesssim \frac{\eps}{\lambda} .    \]
    \end{theorem}


    As compared to \cref{t:KLR}, our result improves the scaling bound and expansion requirement, both by a $\log d$ factor, but we use the stronger $\infty$-expansion condition instead of quantum expansion.
    These quantitative improvements are the key to our optimal sample complexity bounds.

    In order to apply the above scaling result, we show that random instances satisfy $\infty$-expansion. 


    \begin{theorem} \label{t:RandomFrameInfty}
        Given $v_{1}, ..., v_{n} \sim S^{d-1}$ for $n \gtrsim d$ large enough, the input $V$ satisfies $(1-\lambda)$-$\infty$-expansion for $\lambda \geq \Omega(1)$ with probability at least $1 - \exp(-\Omega(n))$. 
    \end{theorem}

    This improves on Theorem 2.10 in \cite{FM20} by showing the stronger $\infty$-expansion condition, as opposed to quantum expansion, as well as reducing the sample requirement by a $\log d$ factor. 
    Our main sample complexity result, \cref{t:mainSampleComplexity}, follows by combining the above ingredients. 

    Our approach to the algorithmic convergence result \cref{t:TylerIterativeAnalysis} also relies on $\infty$-expansion. 
    \cite{FM20} showed that quantum expansion is essentially equivalent to strong convexity of a natural potential function for frame scaling.
    Further, Tyler's iterative approach can be seen as a natural descent method for this potential function, so the convergence bound follows by standard arguments from convex analysis. 

    In \cref{t:InftyImpliesQ} we show that $\infty$-expansion implies quantum expansion. This directly implies that the convergence analysis of \cite{FM20} for Tyler's iterative procedure holds at the optimal sample threshold $n \gtrsim d$.

    \subsection{Organization}

    We first prove our sample complexity result \cref{t:mainSampleComplexity} in \cref{s:sampleComplexity}, assuming our two main technical ingredients. 
    In \cref{s:FrameInfty}, we prove \cref{t:frameMainInfty} which gives a novel scaling result for inputs satisfying $\infty$-expansion. 
    In \cref{s:RandomInfty}, we prove \cref{t:RandomFrameInfty}, showing random random frames satisfy $\infty$-expansion. 
    In \cref{s:InftyImpliesQuantumExpansion}, we prove \cref{t:TylerIterativeAnalysis}, showing convergence of Tyler's iterative procedure.
    We defer some of the technical details from these proofs to \cref{app:frameScaling}, \ref{app:RandomInfty}, and \ref{app:InftyImpliesQuantumExpansion}, respectively. 
    We conclude in \cref{s:conclusion} with some related problems. 

\section{Optimal Sample Complexity via $\infty$-Expansion} \label{s:sampleComplexity}

    In this section, we show \cref{t:mainSampleComplexity} as a straightforward consequence of our two main technical ingredients, which are proved in the following two sections. 
    In order to apply our new frame scaling result in \cref{t:frameMainInfty}, we need the input to be $\eps$-doubly balanced. For this we apply the following standard concentration result for random unit vectors:

    \begin{theorem} [Theorem 5.14 in \cite{FM20}, Theorem 5.39 in Vershynin \cite{vershynin2010introduction}] \label{t:VershyninConcentration}
        Given $n \gtrsim \frac{d}{\eps^{2}}$ uniformly random unit vectors $v_{1}, ..., v_{n} \sim S^{d-1}$, where $\eps$ is at most a small constant, 
        \[  \bigg\| \frac{d}{n} \sum_{j \in [n]} v_{j} v_{j}^{T} - I_{d} \bigg\|_{\op} \leq \eps  \]
        with probability at least $1 - \exp( - \Omega(\eps^{2} n) )$. 
    \end{theorem}

        We can now carry out the argument presented in \cref{ss:techOverview} for our sample complexity bound.

            \vspace{3mm}

    \begin{proof} [Proof of \cref{t:mainSampleComplexity}]
        By \cref{l:EllipticalReductiontoI}, we can assume the underlying distribution is $E(I_{d}, 1)$, 
        so our input frame $V \in \R^{d \times n}$ has columns distributed according to $v_{j} \sim S^{d-1}$. 
        For $n \gtrsim \frac{d}{\eps^{2}}$ large enough, \cref{t:VershyninConcentration} and \cref{t:RandomFrameInfty} imply $V$ is $\eps$-doubly balanced and satisfies $(1-\lambda)$-$\infty$-expansion for $\lambda^{2} \geq \Omega(1) \gtrsim \eps$ with probability $\geq 1 - \exp( - \Omega(\eps^{2} n))$. 
        Therefore, we can apply \cref{t:frameMainInfty} to input $V$ and bound the frame scaling solution $L$
        \[ \| L - I_{d} \|_{\op} \lesssim \frac{\eps}{\lambda} \lesssim \eps , \]
        as $\lambda \geq \Omega(1)$. 
        Finally, we can use \cref{l:EllipticalFrameCorrespondence} to bound the M-Estimator $\hat{\Sigma} = \frac{d  (L^{T} L)^{-1}}{\tr[(L^{T} L)^{-1}]}$:
        \[ \|I_{d} - \Sigma^{1/2} \hat{\Sigma}^{-1} \Sigma^{1/2} \|_{\op} \lesssim \|L - I_{d}\|_{\op} \lesssim \eps ,    \]
        where we used $\Sigma = I_{d}$, as well as the bounds $\frac{1}{d} \tr[(L^{T} L)^{-1}] = 1 + O(\|L - I_{d}\|_{\op})$ and $\|I_{d} - L^{T} L\|_{\op} \lesssim \|L - I_{d}\|_{\op}$ by Taylor approximation.
    \end{proof}

\section{Frame Scaling with $\infty$-Expansion} \label{s:FrameInfty}

The goal of this section is to prove \cref{t:frameMainInfty}, which shows optimal bounds for frame scaling when the input satisfies the $\infty$-expansion condition in \cref{d:frameInftyExpansion}. 
We follow the approach of Kwok et al \cite{KLR}, which proved a slightly weaker bound using the weaker quantum expansion condition. 
Our analysis lifts to the more general operator scaling problem, but we omit this as our statistical setting only involves frame scaling. 

We first collect useful facts about frame scaling from \cite{KLR}. 
Then, we state the main technical lemma where we use $\infty$-expansion. 
The final ingredient is a robustness result, showing $\infty$-expansion is preserved under small scalings. 
The full proof is given in \cref{app:frameScaling}.


    The main character of our analysis is a dynamical system that converges to the frame scaling solution.
    It can be interpreted as a continuous version of the natural Flip-Flop algorithm, so we begin by motivating and defining this procedure. 
    Frame scaling (\cref{d:frameScalingProblem}) involves two conditions, isotropy and equal-norm, each of which is easy to satisfy individually:
    \[ L := (V V^{T})^{-1/2} \implies L V V^{T} L^{T} = I_{d};   \quad
         R := \diag(V^{T} V)^{-1/2} \implies \diag(R^{T} V^{T} V R) = I_{n} . \]
    This suggests the following natural procedure, devised by Gurvits \cite{G04} in the context of operator scaling:
    \begin{equation} \label{eq:FlipFlop}
    V_{t+1} \leftarrow (V_{t} V_{t})^{-1/2} V_{t}, \qquad V_{t+2} \leftarrow V_{t+1} \diag(V_{t+1}^{T} V_{t+1})^{-1/2} .   
    \end{equation} 
    The seminal work of Gurvits et al \cite{GGOW15} showed that this simple algorithm converges to the scaling solution (even for the more general operator scaling problem). Their goal was to prove worst-case time complexity bounds, and they were able to use deep results from geometric invariant theory to analyze the above discrete algorithm.

    In our setting, as well as in \cite{KLR}, we want to prove beyond worst-case bounds for scaling when the input is already nearly doubly balanced. 
    Therefore we study the following continuous process which allows for more refined control of the trajectory, as opposed to the the Flip-Flop algorithm which may take large steps in each iteration.

\begin{definition} [Dynamical System, Def 2.16 in \cite{KLR}] \label{d:frameGradFlow}
    For frame $V \in \R^{d \times n}$ with size $s(V)$ and error matrices $E(V), F(V)$ according to \cref{d:sizeRowCol}, $V_{t}$ is the solution to the following differential equation:
    \[ - \partial_{t} V_{t} = E(V_{t}) V_{t} + V_{t} F(V_{t}) = ( d V_{t} V_{t}^{T} - s(V_{t}) I_{d} ) V_{t} + V_{t} \diag( n V_{t}^{T} V_{t} - s(V_{t}) I_{n} ), \qquad V_{0} = V .     \]
\end{definition}

Intuitively, the first term $E(V) V$ pushes the frame towards isotropy, and the second term $V F(V)$ drives towards equal norms. This can be seen as a continuous and simultaneous version of Flip-Flop.

It turns out that the above dynamical system automatically produces a frame scaling solution! The following result from \cite{KLR} allows us to control the scaling solutions throughout the trajectory. 

    \begin{lemma} [Corollary 3.12 and 3.15 in \cite{KLR}] \label{l:scalingDynamics} \label{l:scalingProdIntBound}
        The solution to the dynamical system in \cref{d:frameGradFlow} produces a frame scaling of the form $V_{t} = L_{t} V R_{t}$, where $L_{t} \in \R^{d \times d}, R_{t} \in \diag(n)$. 
        Further, these scalings can be bounded as 
    \[ \|L_{T} - I_{d}\|_{\op} \leq \exp \bigg( \int_{0}^{T} \|E(V_{t})\|_{\op} \bigg) - 1 ; \quad 
        \|R_{T} - I_{d}\|_{\op} \leq \exp \bigg( \int_{0}^{T} \|F(V_{t})\|_{\op} \bigg) - 1 .    \]    
    \end{lemma}

    Therefore, in order to prove strong scaling bounds, it suffices to show that the error $(E(V), F(V))$ decreases quickly. The following lemma is the key new step in our analysis, where we use $\infty$-expansion condition to show convergence of the error. 

    \begin{lemma} [Informal] \label{p:EFInftyBothDecrease}
        For $\eps$-doubly balanced $V \in \R^{d \times n}$ satisfying $(1-\lambda)$-$\infty$-expansion for $\lambda \gtrsim \eps$,
        \begin{align*}
        -\partial_{t=0} \|(E_{t}, F_{t})\|_{\op} \gtrsim s(V) \cdot \lambda \cdot \|(E_{t}, F_{t})\|_{\op} .
        \end{align*}
    \end{lemma}

    In other words, $\infty$-expansion implies operator norm error decreases at an exponential rate. This is analogous to Prop 3.9 in \cite{KLR} which shows exponential convergence of $\Delta$ using quantum expansion. 
    This justifies \cref{d:frameInftyExpansion} as we are able to use the stronger operator norm condition to prove stronger operator norm bounds. 

    The final step is to show that $\infty$-expansion is maintained throughout the dynamical system: 

  \begin{lemma} [Analogue of Lemma 3.20 in \cite{KLR}] \label{l:frameInftyRobustness}
   Let $V \in \R^{d \times n}$ be $\eps$-doubly balanced for $\eps \leq 1/2$, and consider scaling $U = L V R$ with $\delta := \max\{ \|L - I_{d}\|_{\op}, \|R - I_{n}\|_{\op}\} \leq 1/2$. If $V$ satisfies $(1-\lambda)$-$\infty$-expansion, then $U$ satisfies $(1-\lambda')$-$\infty$-expansion with
   \[ s(U) (1 - \lambda') \leq s(V) (1 - \lambda) + O(\delta)
   \implies \lambda' \geq 1 - \frac{s(V) (1 - \lambda) + O(\delta)}{s(U)}  . \]
  \end{lemma}

  We omit the proof as it is identical to that of Lemma 3.20 in \cite{KLR}, replacing $\|\cdot\|_{F}$ for quantum expansion with $\|\cdot\|_{\op}$ for $\infty$-expansion. 

    \cref{t:frameMainInfty} now follows by combining these ingredients. Namely, \cref{p:EFInftyBothDecrease} shows fast convergence of the error; by \cref{l:scalingProdIntBound}, this implies the scaling cannot become too large; finally, \cref{l:frameInftyRobustness} implies that this maintains $\infty$-expansion. Therefore, we should have fast convergence for all time, which implies strong scaling bounds by \cref{l:scalingProdIntBound}. 
  The formal proof is given in \cref{app:frameScaling}.

\section{$\infty$-Expansion for Random Frames} \label{s:RandomInfty}

In this section, we show random unit vector frames satisfy $\infty$-expansion with high probability. 
Our proof strategy is as follows: first, we reduce $\infty$-expansion to a simpler `pseudorandom' condition involving column subsets;
next, for technical reasons, we first show the pseudorandom condition for Gaussian frames;
finally, we show that this property is maintained after normalization $v_{j} = \frac{g_{j}}{\|g_{j}\|_{2}}$, which proves pseudorandomness for random unit vectors.

We begin by reducing our expansion condition to a simpler condition involving column subsets. Recall that \cref{d:frameInftyExpansion} involves a bound on $\|\sum_{j} y_{j} v_{j} v_{j}^{T}\|_{\op}$ for all $y \perp 1_{n}, \|y\|_{\infty} \leq 1$. 
We instead consider a simpler condition involving only column subsets, i.e. $y \in \{0,1\}^{n}$. 

\begin{definition} [Pseudorandom condition] \label{d:framePseudo}
    Frame $V \in \R^{d \times n}$ is $(\alpha_{\min}, \alpha_{\max}, \beta)$-pseudorandom if
    \[ \forall |B| = \beta n: \quad \beta \frac{\alpha_{\min}}{d} I_{d} \preceq V_{B} V_{B}^{T} \preceq \beta \frac{\alpha_{\max}}{d} I_{d}  .  \]
    We use $(\alpha_{\min}, \beta)$-pseudorandomness to denote that only the lower bound condition is satisfied. 
\end{definition}

    This condition is related to, and slightly weaker than, the pseudorandom condition defined in \cite{KLLR}, which was used to prove strong bounds for the Paulsen problem in frame theory.

We next show that $\beta=1/2$ pseudorandomness and $\infty$-expansion are essentially equivalent. We defer the proof to \cref{app:RandomInfty}.

\begin{lemma} \label{l:pseudoInftyRelation}
Let $V \in \R^{d \times n}$ be an $\eps$-doubly balanced frame.
If $V$ $(\alpha_{\min}, \alpha_{\max}, \frac{1}{2})$-pseudorandom, then $V$ satisfies $(1-\lambda)$-$\infty$-expansion with
\[ s(V) (1-\lambda) \leq \min \{ s(V) (1+\eps) - \alpha_{\min} \quad , \quad \alpha_{\max} - s(V) (1-\eps) \}  .  \]
Conversely, if $V$ satisfies $(1-\lambda)$-$\infty$-expansion, then it is $(\alpha_{\min}, \alpha_{\max}, \frac{1}{2})$-pseudorandom with
\[ s(V) (\lambda - \eps) \leq \alpha_{\min}  \leq  \alpha_{\max} \leq s(V) (2 - (\lambda - \eps)) .    \]
\end{lemma}

We would like to apply concentration to prove the pseudorandom condition for random unit vectors. 
This would require a union bound over all subsets ${[n] \choose \beta n}$, which for $\beta=1/2$ has cardinality $\exp(\Omega(n))$. The failure probability $\exp(- \eps^{2} n)$ in \cref{t:VershyninConcentration} is too high for this approach. 

As a workaround, we note that the uniform distribution $v \sim S^{d-1}$ is equivalent to first sampling a random Gaussian vector $g \sim N(0,I_{d})$, and then normalizing $v \leftarrow g/\|g\|_{2}$. 
It turns out that we can use a slightly stronger concentration bound to show the pseudorandom condition for Gaussians. 

\begin{theorem} \label{p:GaussianPseudo}
    Gaussian frame $G \in \R^{d \times n}$, with entries $G_{ij} \sim N(0,\frac{1}{nd})$, is $(\alpha_{\min} \geq \Omega(1), \alpha_{\max} \leq O(1), \beta = \frac{1}{4})$-pseudorandom according to \cref{d:framePseudo}.
\end{theorem}

We also have that normalization does not affect the pseudorandom property too much. 

\begin{lemma} \label{l:normalizationPseudo}
    For frame $G$ that is $(\alpha_{\min}, \alpha_{\max}, \beta)$-pseudorandom, normalized frame $v_{j} = \frac{g_{j}}{\|g_{j}\|_{2}}$ has size $s(V) = n$ and is $(\alpha_{\min}(V), 2 \beta)$-pseudorandom for 
    \[ \alpha_{\min}(V) \geq s(V) \frac{\alpha_{\min}}{2 \alpha_{\max}} .    \]
\end{lemma}

We prove these results in \cref{app:RandomInfty}.
We can now combine the above ingredients to prove $\infty$-expansion for random unit vectors. 

\begin{proof} [Proof of \cref{t:RandomFrameInfty}]
    \cref{t:GaussianConcentration} implies $(\Omega(1), O(1), \beta=1/4)$-pseudorandomness for Gaussian frame $G \sim N(0, \frac{1}{nd} I_{d} \otimes I_{n})$ with high probability $\geq 1 - \exp( - \Omega(n) )$. 
    $V$ is equivalently distributed as the normalized Gaussian frame $v_{j} \leftarrow \frac{g_{j}}{\|g_{j}\|_{2}}$, 
    so \cref{l:normalizationPseudo} implies that $V$ is $(\alpha_{\min} \geq \Omega(n), \beta = \frac{1}{2})$-pseudorandom. 
    Also, for $\eps$ a small constant, \cref{t:VershyninConcentration} implies $V$ is $\eps$-doubly balanced with probability $\geq 1 - \exp( -\Omega(n))$. 
    Therefore we can apply \cref{l:pseudoInftyRelation} to show $V$ satisfies $(1-\lambda)$-$\infty$-expansion for 
\[ \lambda \geq \frac{\alpha_{\min}(V)}{s(V)} - \eps \geq \Omega(1) ,   \]
where we used $\alpha_{\min} \geq \Omega(n)$, the size is $s(V) = n$, and $\eps$ is a small enough constant.   
\end{proof}

\section{Convergence of Tyler's Iterative Procedure} \label{s:InftyImpliesQuantumExpansion}

In this section we prove \cref{t:TylerIterativeAnalysis}, showing Tyler's iterative procedure has linear convergence to estimator as soon as $n \gtrsim d$. 
This improves on Theorem 1.3 from Franks and Moitra \cite{FM20} by a $\log^{2} d$ factor in terms of the sample requirement.
Concretely, we will show that the key property used in the analysis of \cite{FM20}, quantum expansion, already holds at this lower sample threshold.
This result is optimal up to constants, as the M-estimator does not even exist when $n < d$.


    We begin by motivating and defining Tyler's iterative procedure. 
    In the setting of Elliptical distributions, the goal is to estimate the shape matrix using Tyler's M-Estimator. By \cref{l:EllipticalFrameCorrespondence} this corresponds to the left scaling solution for the sample data. 
    Since the Flip-Flop procedure in \cref{eq:FlipFlop} computes a sequence of iterates converging to the scaling solution, we could use the above correspondence to compute a sequence of estimators converging to the M-Estimator. 

    Tyler's iterative procedure exactly follows two iterations of Flip-Flop while keeping the right scaling implicit. Recall two iterations of Flip-Flop gives
    \[ L_{t+1} \leftarrow ( L_{t} V R_{t}^{2} V^{T} L_{t}^{T} )^{-1/2} L_{t} \quad , \quad R_{t+1} \leftarrow R_{t} \diag( R_{t} V^{T} L_{t}^{T} L_{t} V R_{t})^{-1/2} .   \]
    Now we observe that for any left scaling $L$, there is a naturally induced right scaling $R_{jj} := 1/\|L v_{j}\|_{2}$ which fixes the equal-norm condition. Indeed this is equivalent to the second step of Flip-Flop above. Therefore we can keep this step implicit and update the estimator as 
        \[ \overline{\Sigma}_{t+1} \leftarrow \sum_{j \in [n]} \frac{\overline{\Sigma}_{t}^{-1/2} v_{j} v_{j}^{T} \overline{\Sigma}_{t}^{-1/2}}{v_{j}^{T} \overline{\Sigma}_{t}^{-1} v_{j}}  .  \]

    The main algorithmic result of \cite{FM20} uses the perspective of \emph{geodesic convexity} to analyze Tyler's iterative procedure. 
    In particular, they show that the quantum expansion condition of \cite{KLR} is intimately tied to the geodesic version of strong convexity. 

    The result below follows from the arguments in \cite{FM20}: 
    
    \begin{theorem} [Section 4 of \cite{FM20}] \label{t:FMAlgoExtract}
        Consider input $x_{1}, ..., x_{n} \sim E(\Sigma, u)$ with Tyler's M-Estimator $\hat{\Sigma}$. Let $U := \{u_{1}, ..., u_{n}\}$ be the `normalized' frame, with $u_{j} := \frac{\hat{\Sigma}^{-1/2} x_{j}}{\|\hat{\Sigma}^{-1/2} x_{j}\|_{2}}$. If $U$ is a $(1-\Omega(1))$-quantum expander according to \cref{d:frameQuantumExpansion}, then 
        Tyler's iterative procedure satisfies approximation $\|I_{d} - \hat{\Sigma}^{1/2} \overline{\Sigma}_{(T)}^{-1} \hat{\Sigma}^{-1/2} \|_{F} \leq \delta$ within iteration $T \lesssim |\log\det(\hat{\Sigma})| + d + \log(1/\delta)$. 
    \end{theorem}
    \begin{proof} [Proof Sketch]
        For input $X$, the inverse estimator $\hat{\Sigma}^{-1}$ is the optimizer of the following capacity function
        \[ f_{X}(Z) := \frac{d}{n} \sum_{j \in [n]} \log \langle v_{j}, Z v_{j} \rangle - \log \det(Z) .   \]
        By Theorem 4.1 in \cite{FM20}, if $U$ is a $(1-\Omega(1))$-quantum expander, then $f_{X}$ is $\Omega(1)$ geodesically strongly convex on a large neighborhood of the optimizer $\hat{\Sigma}^{-1}$. 
        Lemma 4.2 in \cite{FM20} shows that Tyler's iterative procedure is a descent method for $f_{X}$ which starts at $\overline{\Sigma}_{(0)} := I_{d}$ and reaches this strongly convex neighborhood within $O(|\log \det(\hat{\Sigma})| + d)$ iterations. 
        From this point, Lemma 4.3 in \cite{FM20} shows linear convergence to the optimum via geodesic strong convexity. 
    \end{proof}

    The key technical contribution in Theorem 2.9 of \cite{FM20} is to show random unit vector frames are quantum expanders above sample threshold $n \gtrsim d \log d$. 
    In \cref{t:RandomFrameInfty}, we showed that random frames satisfy the $\infty$-expansion condition. 
    Our refined result in \cref{t:frameMainInftyR} in fact shows that the frame scaling solution also satisfies $\infty$-expansion.
    In \cref{app:InftyImpliesQuantumExpansion} we show that $\infty$-expansion implies quantum expansion.

    \begin{theorem} \label{t:InftyImpliesQ}
        For doubly balanced frame $V \in \R^{d \times n}$, if $V$ satisfies $(1-\lambda)$-$\infty$-expansion, then it is also a $(1-\Omega(\lambda^{2}))$-quantum expander. 
    \end{theorem}

    We can now combine this result with the arguments of \cite{FM20} to prove linear convergence of Tyler's iterative procedure at the optimal sample threshold. 

    \vspace{3mm}

    \begin{proof} [Proof of \cref{t:TylerIterativeAnalysis}]
        Let $X \in \R^{d \times n}$ be generated according to $E(\Sigma, u)$. By \cref{l:EllipticalReductiontoI}, frame $V \in \R^{d \times n}$ has columns $v_{j} := \frac{\Sigma^{-1/2} x_{j}}{\|\Sigma^{-1/2} x_{j}\|_{2}}$ distributed according to $E(I_{d}, 1)$. 
        Similarly, by \cref{l:EllipticalFrameCorrespondence}, scalings $L := \hat{\Sigma}^{-1/2}, R_{jj} := 1/\|\hat{\Sigma}^{-1/2} x_{j}\|_{2}$ gives frame scaling solution $U := L V R$ for input $V$. 
        Therefore, in order to apply \cref{t:FMAlgoExtract}, we need to show that the frame scaling solution for random unit vector input satisfies quantum expansion. 

        Following the proof of \cref{t:mainSampleComplexity}, by \cref{t:VershyninConcentration} and \cref{t:RandomFrameInfty}, we have that $V$ is $\eps$-doubly balanced and satisfies $(1-\lambda)$-$\infty$-expansion for $\lambda \geq \Omega(1)$. 
        Therefore, by \cref{t:frameMainInftyR}(3), we have that the frame scaling solution $U$ satisfies $(1-\Omega(1))$-$\infty$-expansion. 
        Finally, \cref{t:InftyImpliesQ} tells us that $U$ is also a $(1-\Omega(1))$-quantum expander. 
        Also note $|\log \det(\hat{\Sigma})| \approx |\log\det(\Sigma)|$ by the correspondence to the scaling solution (\cref{l:EllipticalFrameCorrespondence}) and \cref{t:frameMainInftyR}(2). 
        The iteration bound follows by \cref{t:FMAlgoExtract}. 
    \end{proof}

    \section{Conclusion} \label{s:conclusion}

In this work, we showed optimal sample complexity bounds for Tyler's M-Estimator for the shape matrix of Elliptical distributions. 
These bounds are tight (up to constant factors), as they match known lower bounds for the special case of Gaussian covariance estimation. 
We also recover the algorithmic analysis of Franks and Moitra \cite{FM20}, showing linear convergence of Tyler's iterative procedure with optimal sample threshold $n \gtrsim d$. 

Our proof follows the connection to frame scaling, as described in \cite{FM20}. 
We identify a new `pseudorandom' condition for frames, $\infty$-expansion, and show it is satisfied by random frames with high probability. 
Next, we use the $\infty$-expansion condition to prove a new result for frame scaling. 
We believe this result could be of independent interest.
We highlight two problems which could potentially benefit from similar techniques. 

    The first is the Paulsen problem: given $\eps$-doubly balanced frame $U \in \R^{d \times n}$, find a nearby frame $V \in \R^{d \times n}$ that is doubly balanced. 
    This was a major problem in frame theory \cite{CC13}, open for nearly two decades despite significant attention.
    The main conjecture was to show the distance was independent of $n$.
    In \cite{KLLR}, the authors were able to affirm this conjecture, showing a distance bound of $\text{poly}(d) \cdot \eps$ using a smoothed analysis approach to frame scaling. Indeed, their analysis involved a pseudorandom property which is very similar to our results in \cref{s:FrameInfty}. 
    \cite{HM19} improved the distance bound to $d \cdot \eps$, which is a factor $d$ greater than the known lower bound $\eps$. 
    In upcoming work, we are able to use techniques similar to this paper to close this gap and show optimal distance $\eps$ for the Paulsen problem.


The second problem is the matrix and tensor normal model in statistics. We are given tensor data $X \in \R^{d_{1}} \otimes ... \otimes \R^{d_{k}}$ from a Gaussian distribution $N(0,\Sigma)$, with the promise that the covariance matrix has a matching form $\Sigma = \Sigma_{1} \otimes ... \otimes \Sigma_{k}$. In \cite{MLE}, the authors show strong sample complexity bounds for this problem. 
We believe that techniques related to $\infty$-expansion can be used to give optimal sample complexity and error bounds for the tensor normal model.

\section{Bibliography}

  \bibliographystyle{plain}
\bibliography{refs.bib}

\appendix


\section{Proof of \cref{t:frameMainInfty}} \label{app:frameScaling}

    Here we prove \cref{t:frameMainInfty} using $\infty$-expansion. We show the following more precise result:

      \begin{theorem} \label{t:frameMainInftyR}
    Consider frame $V \in \R^{d \times n}$ that is $\eps$-doubly balanced and satisfies $(1-\lambda)$-$\infty$-expander with $\lambda^{2} \gtrsim \eps$. If $(V_{t}, L_{t}, R_{t})$ is the dynamical system given in \cref{d:frameGradFlow} and \cref{l:scalingDynamics}, 
  \begin{enumerate}
    \item $(L_{\infty}, R_{\infty}) := \lim_{t \to \infty} (L_{t}, R_{t})$ is the frame scaling solution for $V$ according to \cref{d:frameScalingProblem};
    \item The scaling solution satisfies $ \| L_{\infty} - I_{d} \|_{\op}, \|R_{\infty} - I_{n}\|_{\op} \lesssim \frac{\eps}{\lambda}$;
    \item $V_{\infty} := \lim_{t \to \infty} V_{t} = L_{\infty} V R_{\infty}$ satisfies $(1 - \lambda/2)$-$\infty$-expansion, and has size $s(V_{\infty}) \geq s(V)( 1 - O(\frac{\eps^{2}}{\lambda}) )$. 
    \end{enumerate}
    \end{theorem}

    Our plan is to show exponential convergence of the error matrices through gradient flow. 
    We begin by explicitly calculating how the error matrices change in the dynamical system. The following expression is implicit in the proof of Proposition 3.2 in \cite{KLR}. 

    \begin{fact} \label{eq:rowcolChangeR}
    For frame $V \in \R^{d \times n}$, let $V_{t}$ be the solution to the dynamical system as given in \cref{d:frameGradFlow}. Then for any fixed $x \in \R^{d}$ and $j \in [n]$, 
    \[  -\partial_{t=0} \langle x x^{T}, V V^{T} \rangle = 2 \langle x x^{T}, E V V^{T} + V F V^{T} \rangle; 
    \quad -\partial_{t=0} \|v_{j}\|_{2}^{2} = 2 ( F_{jj} \|v_{j}\|_{2}^{2} + \langle E, v_{j} v_{j}^{T} \rangle ).   \]
    \end{fact}
    \begin{proof}
    For the first expression, we compute
    \begin{align*}
        -\partial_{t=0} \langle x x^{T}, V_{t} V_{t}^{T} \rangle &= \langle x x^{T}, (E V + V F) V^{T} + V (EV + VF)^{T} \rangle = 2 \langle x x^{T}, E V V^{T} + V F V^{T} \rangle , 
    \end{align*}
    where we used the product rule and substitute \cref{d:frameGradFlow} for the derivative, and in the last step we used that $E = d \cdot V V^{T} - s(V) I_{d}$ commutes with $V V^{T}$. 
    For the second, a similar calculation gives
    \begin{align*}
        -\partial_{t=0} V_{t}^{T} V_{t} &= (E V + V F)^{T} V + V^{T} (EV + VF) = 2 V^{T} E V + F V^{T} V + V^{T} V F . 
    \end{align*}
    The second expression follows by considering the $j$-th diagonal entry of the above expression. 
    \end{proof}
    
    Recall that $E = d V V^{T} - s I_{d}$ and $F = \diag(n V^{T} V - s I_{n})$ according to \cref{d:sizeRowCol}, and we want to show that these errors decrease under $\infty$-expansion.
    For illustration, let $x \in S^{d-1}$ be the top eigenvector $E x = \|E\|_{\op} x$, and note this implies $d \cdot V V^{T} x = (s + \|E\|_{\op}) x$. 
    Then the first term above is
    \[ d \cdot \langle x x^{T}, E V V^{T} \rangle = \|E\|_{\op} \langle x x^{T}, d V V^{T} \rangle = \|E\|_{\op} (s + \|E\|_{\op}) \approx s \|E\|_{\op} ,   \]
    where we assume $\|E\|_{\op} \ll s$ for simplicity. 
    This tells us that for top eigenvector $x$,
    the $E$ term in \cref{eq:rowcolChangeR} decreases the quadratic form $\langle x x^{T}, V V^{T} \rangle$. 
    Intuitively, this term pushes the error towards balanced. 
    
    The other term involves $V F V^{T}$ for $\tr[F] = 0$. We can bound this using $\infty$-expansion: 
    \[ |d \langle x x^{T}, V F V^{T} \rangle | \leq \|x\|_{2}^{2} \|d V F V^{T}\|_{\op} \leq s (1 - \lambda) \|F\|_{\op} , \]
    where the first step was by definition of operator norm, and in the second step we applied \cref{d:frameInftyExpansion} to $F$. 
    This implies that this second term contributes only a small competing force for the change in \cref{eq:rowcolChangeR}, and so intuitively the error should decrease quickly. 
    We formalize this intuition below. We will also need the following lemma to control how the size changes. 
    


    \begin{lemma} [Lemma 2.17 in \cite{KLR}] \label{l:sizeChange}
    The change in size of $V$ for the dynamical system in \cref{d:frameGradFlow} is
    \[ -\partial_{t} s(V_{t}) = 2 \Delta(V) 
    = \frac{2}{d} \|E(V)\|_{F}^{2} + \frac{2}{n} \|F(V)\|_{F}^{2}  . \]
    \end{lemma}


    \begin{lemma} \label{l:EFInftyBound}
    Given frame $V \in \R^{d \times n}$ with $V_{t}$ the solution to the dynamical system in \cref{d:frameGradFlow}, if $V$ satisfies $(1-\lambda)$-$\infty$-expansion at time $t=0$, then
        \begin{align*} 
        -\frac{1}{2} \partial_{t=0} \|E(V_{t})\|_{\op} & \geq s(V) ( \|E\|_{\op} - (1-\lambda) \|F\|_{\op} ) - \|(E,F)\|_{\op}^{2} ;  
        \\  -\frac{1}{2} \partial_{t=0} \|F(V_{t})\|_{\op} & \geq s(V) ( \|F\|_{\op} - \|E\|_{\op}) - \|(E,F)\|_{\op} (\|E\|_{\op} + \|F\|_{\op} ) .  
        \end{align*}
    \end{lemma}    
    \begin{proof}
        We begin with the $E$ term. First assume $\|E\|_{\op} = \langle x x^{T}, E \rangle$ for $x \in S^{d-1}$, so we compute 
        \[ -\partial_{t=0} \|E(V_{t})\|_{\op} 
        = -\partial_{t=0} \langle x x^{T}, d V_{t} V_{t}^{T} - s(V_{t}) I_{d} \rangle 
        = 2 d \langle x x^{T}, E V V^{T} + V F V^{T} \rangle - 2 \Delta(V) ,     \]
        where in the first step we used the generalized Envelope Theorem of \cite{milgrom2002envelope} to write the derivative of $\|E\|_{\op} = \max_{x \in S^{d-1}} |\langle x x^{T}, E \rangle|$, 
        and in the final step we used \cref{eq:rowcolChangeR} for the first term and \cref{l:sizeChange} for the change in size. We continue
        \begin{align*} 
        - \frac{1}{2} \partial_{t=0} \|E(V_{t})\|_{\op} 
        & \geq \|E\|_{\op} (s(V) + \|E\|_{\op}) - d \|V F V^{T}\|_{\op} - ( \|E\|_{\op}^{2} + \|F\|_{\op}^{2} )   
        \\ & \geq s(V) \|E\|_{\op} - s(V) (1 - \lambda) \|F\|_{\op}^{2} - \|F\|_{\op}^{2} , 
        \end{align*}
        where we used that $x$ is a unit eigenvector of $E$ and $d V V^{T}$, with eigenvalue $\|E\|_{\op}$ and $s(V) + \|E\|_{\op}$ respectively, and we bounded $\Delta(V) \leq \|E\|_{\op}^{2} + \|F\|_{\op}^{2}$ using \cref{f:matrixdbSmallGrad}. 

        For the other case $\|E\|_{\op} = -\langle x x^{T}, E \rangle$ with $x \in S^{d-1}$, we can use the same argument to bound
        \begin{align*} 
        -\frac{1}{2} \partial_{t=0} & \|E(V_{t})\|_{\op} 
        = \frac{1}{2} \partial_{t=0} \langle x x^{T}, d V_{t} V_{t}^{T} - s(V_{t}) I_{d} \rangle 
        = - d \langle x x^{T}, E V V^{T} + V F V^{T} \rangle + \Delta(V)    
        \\ & \geq \|E\|_{\op} (s(V) - \|E\|_{\op}) - d \|V F V^{T}\|_{\op}
        \geq s(V) (\|E\|_{\op} - (1 - \lambda) \|F\|_{\op}) - \|E\|_{\op}^{2} ,
        \end{align*}
        where we used that $x$ is an eigenvector of $d V V^{T}$ with eigenvalue $s(V) - \|E\|_{\op}$, and $\Delta(V) \geq 0$.  
        The result follows by combining the two cases. 

        For the $F$-bound, we follow the same steps as above without applying $\infty$-expansion. So first, if $\|F\|_{\op} = n \|v_{j}\|_{2}^{2} - s(V)$ for some $j \in [n]$, then 
        \[ -\partial_{t=0} \|F(V_{t})\|_{\op} 
        = -\partial_{t=0} (n \|V_{t} e_{j}\|_{2}^{2} - s(V_{t}))
        = 2 n ( F_{jj} \|v_{j}\|_{2}^{2} + \langle v_{j}, E v_{j} \rangle ) - 2 \Delta(V) ,     \]
        where the first step is again by the Envelope Theorem \cite{milgrom2002envelope} applied to $\|F\|_{\op} = \max_{j \in [n]} |F_{jj}|$ for diagonal $F$,
        and in the final step we used \cref{eq:rowcolChangeR} for the first term and \cref{l:sizeChange} for the change in size. We continue
        \begin{align*} 
        - \frac{1}{2} \partial_{t=0} \|F(V_{t})\|_{\op} 
        & \geq n \|v_{j}\|_{2}^{2} (\|F\|_{\op}  - \|E\|_{\op} ) - ( \|E\|_{\op}^{2} + \|F\|_{\op}^{2} )   
        \\ & \geq s(V) (\|F\|_{\op} - \|E\|_{\op}) - \|E\|_{\op} ( \|E\|_{\op} + \|F\|_{\op}), 
        \end{align*}
        where we used $n \|v_{j}\|_{2}^{2} = s(V) + \|F\|_{\op}$ by assumption on $j$, bounded $\langle v_{j}, E v_{j} \rangle \leq \|E\|_{\op} \|v_{j}\|_{2}^{2}$, and used \cref{f:matrixdbSmallGrad} to bound $\Delta(V)$. 

        For the other case, $\|F\|_{\op} = -(n\|v_{j}\|_{2}^{2} - s(V))$, we can use the same argument to show
        \begin{align*} 
        -\frac{1}{2}\partial_{t=0} \|F(V_{t})\|_{\op} 
        & = \frac{1}{2} \partial_{t=0} (n \|V_{t} e_{j}\|_{2}^{2} - s(V_{t}))
        = - n ( F_{jj} \|v_{j}\|_{2}^{2} + \langle v_{j}, E v_{j} \rangle ) + \Delta(V) 
        \\ & \geq n \|v_{j}\|_{2}^{2} (\|F\|_{\op} - \|E\|_{\op} ) 
        \geq (s(V) - \|F\|_{\op}) (\|F\|_{\op} - \|E\|_{\op}) .
        \end{align*}
        By \cref{d:DeltaOpError} $\|(E,F)\|_{\op} = \max\{ \|E\|_{\op}, \|F\|_{\op}\}$, so the result follows by combining both cases. 

    \end{proof}

    The above lemma shows that the convergence of $E$ and $F$ depend on the relative sizes of $\|E\|_{\op}, \|F\|_{\op}$. 
    By some case analysis, we can show one of the errors always decreases. 

    \begin{lemma} \label{p:EFInftyBothDecreaseR}
        Let frame $V \in \R^{d \times n}$ be $\eps$-doubly balanced and satisfy $(1-\lambda)$-$\infty$-expansion for $1 \geq \lambda \geq \eps$. Then for any $\delta \in [0,1)$:
        \begin{align*}
            (1+\delta) \|E\|_{\op} \geq \|F\|_{\op} \quad & \implies \qquad - \partial_{t=0} \log \|E(V_{t})\|_{\op} \geq 2 s(V) ( \lambda - \delta - \eps ); 
            \\  \|E\|_{\op} \leq (1-\delta) \|F\|_{\op} \quad & \implies \qquad - \partial_{t=0} \log \|F(V_{t})\|_{\op} \geq 2 s(V) ( \delta - 2 \eps ) . 
        \end{align*}
    \end{lemma}
    \begin{proof}
        In the first case $(1+\delta) \|E\|_{\op} \geq \|F\|_{\op}$, we can bound the change in $E$ as
        \begin{align*}
        - \frac{1}{2} \partial_{t=0} \log \|E(V_{t})\|_{\op} 
        & = \frac{-\partial_{t=0} \|E(V_{t})\|_{\op}}{2 \|E(V)\|_{\op}}
        \geq s(V) \bigg( 1 - (1-\lambda) \frac{\|F\|_{\op}}{\|E\|_{\op}} \bigg) - \frac{\|(E,F)\|_{\op}^{2}}{\|E\|_{\op}}
        \\ & \geq s(V) (1 - (1-\lambda)(1+\delta)) - s(V) \eps (1+\delta) 
        \geq s(V) (\lambda - \delta - \eps) , 
        \end{align*}
        where we used the $E$-bound from \cref{l:EFInftyBound}, for the third step we used the case assumption to bound $\|F\|_{\op} \leq \|(E,F)\|_{\op} \leq (1+\delta) \|E\|_{\op}$, and the $\eps$-doubly balanced condition to bound $\|(E,F)\|_{\op} \leq s(V) \eps$ according to \cref{d:DeltaOpError}, 
        and in the last step we used $1 \geq \lambda \geq \eps$. 

        Similarly, in the other case $\|E\|_{\op} \leq (1-\delta) \|F\|_{\op}$, we can bound $F$ as 
        \begin{align*} 
        - \frac{1}{2} \partial_{t=0} \log \|F(V_{t})\|_{\op} & 
        = \frac{-\partial_{t=0} \|F(V_{t})\|_{\op}}{2 \|F\|_{\op}} \geq s(V) \bigg( 1 - \frac{\|E\|_{\op}}{\|F\|_{\op}} \bigg) - \frac{\|(E,F)\|_{\op}}{\|F\|_{\op}} (\|E\|_{\op} + \|F\|_{\op})
        \\ & \geq s(V) (1 - (1-\delta)) - s(V) \eps (1 + (1-\delta))
        \geq s(V) ( \delta - \eps(2 - \delta )) , 
        \end{align*}
        where we used the $F$-bound from \cref{l:EFInftyBound}, and for the third step we used the case assumption $\|E\|_{\op} \leq (1-\delta) \|F\|_{\op}$ and the $\eps$-doubly balanced condition.
    \end{proof}
    


    Choosing $\delta$ appropriately and combining this analysis with \cref{l:frameInftyRobustness} gives our result.

\begin{proof} [Proof of \cref{t:frameMainInftyR}]
    We assume $s(V) = 1$ since the doubly balanced and $\infty$-expansion condition are homogenous. 
    We require the following technical assumptions:
    \begin{enumerate}
        \item[(a)] $1 \geq s(V_{t}) \geq 1 - \eps$;
        \item[(b)] $V_{t}$ is $\frac{5}{3} \eps$-doubly balanced.
        \item[(c)] $V_{t}$ is $(1-\lambda/2)$-$\infty$-expander
    \end{enumerate}
    Note that all three conditions are strictly satisfied at time $t=0$. Let $T$ be the first time some assumption fails. 
    Our plan is to show exponential convergence of the error 
    \[ \forall t \in [0,T]: \qquad \|(E_{t},F_{t})\|_{\op} \lesssim \eps e^{-\Omega(\lambda t)}. \]
    We use this convergence to show $T=\infty$ and the conclusions of the theorem. 


    For $\delta \in [0,1)$ chosen later, consider potential function 
    \[ h(t) := \max \{ (1+\delta) \|E_{t}\|_{\op}, \|F_{t}\|_{\op} \} ,    \]
    where we use shorthand $E_{t} := E(V_{t}), F_{t} := F(V_{t})$.
    For a given time $t \in [0,T]$, 
    if the first term is larger, then we have $(1+\delta) \|E_{t}\|_{\op} \geq \|F_{t}\|_{\op}$, so we apply the first line of \cref{p:EFInftyBothDecreaseR}: 
    \[ - \partial_{t} \log h(t) = - \frac{\partial_{t} \|E(V_{t})\|_{\op}}{\|E(V_{t})\|_{\op}} \geq 2 s(V_{t}) (\frac{\lambda}{2} - \delta - \frac{5}{3} \eps)   ,   \]
    where in the last step we used assumption (b) and (c) that $V_{t}$ is $\frac{5}{3} \eps$-doubly balanced and satisfies $(1-\lambda/2)$-$\infty$-expansion conditions with $1 \geq \frac{\lambda}{2} \geq \frac{5}{3} \eps$. 
    
    Otherwise, we have $(1+\delta) \|E\|_{\op} \leq \|F\|_{\op} \implies \frac{\|E\|_{\op}}{\|F\|_{\op}} \leq (1+\delta)^{-1} = 1 - \frac{\delta}{1+\delta}$. Therefore we can apply the second line of \cref{p:EFInftyBothDecreaseR}: 
    \[ - \partial_{t} \log h(t) = - \frac{\partial_{t} \|F(V_{t})\|_{\op}}{\|F(V_{t})\|_{\op}} \geq 2 s(V_{t}) ( \frac{\delta}{1+\delta} - 2 \frac{5}{3} \eps)  ,    \]
    where again we used assumption (c) that $V_{t}$ is $\frac{5}{3} \eps$-doubly balanced. 

    Choosing $\delta = \lambda/4 \leq 1/4$ to balance the two cases, we can bound the potential function 
    \[ - \partial_{t} \log h(t) \geq 2 s(V_{t}) \min\{ \frac{\lambda}{2} - \frac{\lambda}{4} - 2 \eps, \frac{\lambda}{4(1 + \delta)} - \frac{10}{3} \eps \} \geq \frac{\lambda}{3}   ,    \]
    where we used assumption (a) $s(V_{t}) \geq 1 - \frac{\eps}{4}$, as well as $\delta \leq 1/4$, and $\lambda \gtrsim \eps$ large enough.

    Therefore, we have exponential convergence for the error:
    \[ \log \max\{ (1+\delta) \|E_{t}\|_{\op}, \|F_{t}\|_{\op} \} = \log h(t) = \log h(0) + \int_{\tilde{t}=0}^{t} \partial_{\tilde{t}} \log h(\tilde{t}) 
    \leq \log (1+\delta) \eps  -\lambda t /3  , \]
    where the first step was by definition of $h(t)$, the second was by the fundamental theorem of calculus, and in the final step we bounded the first term $h(0) \leq \eps (1+\delta)$ as $V$ is $\eps$-doubly balanced with $s(V) = 1$, and for the integral we used that $\partial_{t} \log h(t) \leq - \lambda/3$ as shown above. Rearranging gives 
    \begin{align} \label{eq:EFConvergence}
        \|E_{t}\|_{\op} \leq \eps e^{-\lambda t / 3}, \quad \|F_{t}\|_{\op} \leq (1+\delta) \eps e^{-\lambda t/3} \leq \frac{5}{4} \eps e^{-\lambda t/3} . 
    \end{align}

    We can now use this exponential convergence to prove the conclusions of the theorem. We begin by bounding the scaling up to time $T$:
    \begin{align} \label{eq:frameInftyScalingBound}
        \|L_{T} - I_{d} \|_{\op} \leq \exp( \int_{0}^{T} \|E_{t}\|_{\op} ) - 1 \leq \exp( \eps \int_{0}^{T} e^{-\lambda t/3} ) - 1 \leq \exp( \frac{3 \eps}{\lambda} ) - 1 \leq \frac{4 \eps}{\lambda} ,      \end{align}
    where the first step was by \cref{l:scalingProdIntBound}, in the second step we plugged in the convergence bound of \cref{eq:EFConvergence}, and the final steps were by Taylor approximation $e^{x} = 1 + O(x)$ along with the assumption $\lambda^{2} \gtrsim \eps$ large enough. The same argument gives the bound on right scaling
    \[ \|R_{T} - I_{n} \|_{\op} \leq \exp( \int_{0}^{T} \|F_{t}\|_{\op} ) - 1 \leq \frac{5 \eps}{\lambda} .  \]

    By \cref{l:frameInftyRobustness}, this implies $V_{T}$ satisfies $(1-\lambda')$-$\infty$-expansion with
    \begin{align} \label{eq:frameInftyBound} \lambda' \geq 1 - \frac{s(V) (1 - \lambda) + O(\eps/\lambda)}{s(V_{T})} \geq \lambda - O(\lambda + \eps) > \lambda/2 ,  
    \end{align}
    where we used assumptions (a) $s(V_{T}) \geq 1 - \eps$ in the second step, and $\lambda^{2} \gtrsim \eps$ in the final step.  

    We can also bound the change in size up to time $T$:
    \begin{align} \label{eq:frameInftySizeBound} 
    s(V) - s(V_{T}) 
    & = - \int_{0}^{T} \partial_{t} s(V_{t}) 
    = 2 \int_{0}^{T} \Delta(V_{t}) 
    \leq 2 \int_{0}^{T} (\|E_{t}\|_{\op}^{2} + \|F_{t}\|_{\op}^{2}) 
    \\ & \leq 6 \eps^{2} \int_{0}^{T} e^{-2 \lambda t / 3} 
    < \frac{6 \eps^{2}}{2 \lambda/3} < \eps ,   
    \end{align}
    where the first step was by fundamental theorem of calculus, in the second step we used \cref{l:sizeChange} for $\partial_{t} s$, in the third we used \cref{f:matrixdbSmallGrad} to bound $\Delta$, in the fourth we applied convergence from \cref{eq:EFConvergence}, and in the final inequality we used our assumption that $\lambda \gtrsim \eps$ large enough. 

    Now Assume for contradiction that $T < \infty$, so one of the assumptions (a,b,c) fails at time $T$. Assumption (a) $1 \geq s(V_{T}) > 1 - \eps$ is strictly satisfied by \cref{eq:frameInftySizeBound}. We can also bound
    \[  \frac{\|(E_{T}, F_{T})\|_{\op}}{s(V_{T})} < \frac{ 5 \eps/4}{1 - 9 \eps^{2}/\lambda} < \frac{5}{3} \eps , \]
    where we used \cref{eq:EFConvergence} to upper bound the error and the calculation above to lower bound the size. By \cref{d:DeltaOpError}, this implies $V_{T}$ is $\eps'$-doubly balanced for $\eps' < \frac{5}{3} \eps$, so assumption (b) is strictly satisfied. Finally, \cref{eq:frameInftyScalingBound} shows that $V_{T}$ satisfies $(1-\lambda')$-$\infty$-expansion for $\lambda' > \lambda/2$, so assumption (c) is also strictly satisfied. 
    By continuity, (a,b,c) must still be satisfied for some $T' > T$. But this is a contradiction, so $T=0\infty$ and the above analysis holds for all time. 
    
    We can now show the conclusions of the theorem. For item (1), note
    \[ \lim_{t \to \infty} \|(E_{t}, F_{t})\|_{\op} \lesssim \lim_{t \to \infty} \eps e^{-\lambda t/3} = 0 , \]
    i.e. $V_{\infty}$ is doubly balanced. 
    By \cref{l:scalingProdIntBound}, this implies the scaling limits $(L_{\infty}, R_{\infty})$ exist and are the frame scaling solution for $V$ according to \cref{d:frameScalingProblem}. 
    Item (2) follows from the calculation in \cref{eq:frameInftyScalingBound} applied at $T=\infty$, and similarly item (3) follows from \cref{eq:frameInftyBound}, (\ref{eq:frameInftySizeBound}). 
\end{proof}

\section{Proof of $\infty$-Expansion for Random Frames} \label{app:RandomInfty}

In this section, we prove the technical details required for our $\infty$-expansion result in \cref{t:RandomFrameInfty}. 

We first restate and prove \cref{l:pseudoInftyRelation}, showing $\beta=1/2$ pseudorandomness and $\infty$-expansion are essentially equivalent. 
We will consider pseudorandomness for other $\beta$ later in this section. 

\begin{lemma} \label{l:pseudoInftyRelationR}
Let $V \in \R^{d \times n}$ be an $\eps$-doubly balanced frame.
If $V$ $(\alpha_{\min}, \alpha_{\max}, \frac{1}{2})$-pseudorandom, then $V$ satisfies $(1-\lambda)$-$\infty$-expansion with
\[ s(V) (1-\lambda) \leq \min \{ s(V) (1+\eps) - \alpha_{\min} \quad , \quad \alpha_{\max} - s(V) (1-\eps) \}  .  \]
Conversely, if $V$ satisfies $(1-\lambda)$-$\infty$-expansion, then it is $(\alpha_{\min}, \alpha_{\max}, \frac{1}{2})$-pseudorandom with
\[ s(V) (\lambda - \eps) \leq \alpha_{\min}  \leq  \alpha_{\max} \leq s(V) (2 - (\lambda - \eps)) .    \]
\end{lemma}

    \begin{remark}
    We always assume $\beta n$ is an integer for simplicity. Technically, the condition for non-integer $\beta n$ should involve a fractional coordinate. For example if $\beta n = k + z$ where $z \in (0,1)$ then $\beta$-pseudorandomness should be the condition
    \[ \beta \frac{\alpha_{\min}}{d} I_{d} \preceq V_{B} V_{B}^{T} + z v_{j} v_{j}^{T} \preceq \beta \frac{\alpha_{\max}}{d} I_{d}    \]
    for all $|B| = k, j \not\in B$. For simplicity we will focus on the integer case as the fractional definition only differs by lower order terms.  
\end{remark}

\begin{proof} [Proof of \cref{l:pseudoInftyRelationR}]
    It can be shown by majorization that the vertices of the polytope
    \[ P := 1_{n}^{\perp} \cap B_{\infty} = \{y \in \R^{n} \mid \langle y, 1_{n} \rangle = 0, \|y\|_{\infty} \leq 1 \}   \]
  are of the form $1_{A} - 1_{B}$ for disjoint sets $|A| = |B| = \lfloor n/2 \rfloor$. For simplicity, we assume $n$ is even so these vertices can be rewritten as $1_{n} - 2 1_{B} = 2 1_{A} - 1_{n}$. This implies that for any $a \in \R^{n}$, 
    \begin{align} \label{eq:PVertices}
        \sup_{y \in P} \langle a, y \rangle = 2 \max_{A} \langle a, 1_{A} \rangle - \langle a, 1_{n} \rangle = \langle a, 1_{n} \rangle - 2 \min_{B} \langle a, 1_{B} \rangle
    \end{align}  
    where $|A| = |B| = n/2$. We relate $\infty$-expansion and pseudorandomness by applying the above for $a_{j} := \langle x, v_{j} \rangle^{2}$ with $x \in S^{d-1}$. 
    
    In particular, let $x \in S^{d-1}, y \in P$ satisfy $\frac{s(V) (1-\lambda)}{d} = | \langle x x^{T}, V Y V^{T} \rangle|$. If $V$ is $(\alpha_{\min}, \alpha_{\max}, \frac{1}{2})$-pseudorandom according to \cref{d:framePseudo} and $\eps$-doubly balanced according to \cref{d:DeltaOpError}, then 
    \[ \frac{s(V) (1-\lambda)}{d} \leq 2 \max_{A} \langle x x^{T}, V_{A} V_{A}^{T} \rangle - \langle x x^{T}, V V^{T} \rangle \leq \frac{\alpha_{\max}}{d} - \frac{s(V)(1-\eps)}{d} ,  \]
    where we used \cref{eq:PVertices} in the first step, and in the last step we upper bounded $V_{A} V_{A}^{T}$ using $\alpha_{\max}$ and the other term using the $\eps$-doubly balanced condition. By the same argument we get
    \[ \frac{s(V) (1-\lambda)}{d} \leq \langle x x^{T}, V V^{T} \rangle  - 2 \min_{B} \langle x x^{T}, V_{B} V_{B}^{T} \rangle \leq \frac{s(V)(1+\eps)}{d} - \frac{\alpha_{\min}}{d}  ,  \]
    where we instead lower bounded $V_{B} V_{B}^{T}$ using $\alpha_{\min}$. 

    Conversely, let $x \in S^{d-1}, |B| = n/2$ satisfy $2 \langle x x^{T}, V_{B} V_{B}^{T} \rangle = \frac{\alpha_{\max}}{d}$. If $V$ satisfies $(1-\lambda)$-$\infty$-expansion, then for $y := 2 1_{B} - 1_{n}$
    \[ \frac{\alpha_{\max}}{d} = \langle x x^{T}, (2 V_{B} V_{B}^{T} - V V^{T}) + V V^{T} \rangle \leq \| V Y V^{T} \|_{\op} + \|V V^{T}\|_{\op} \leq \frac{s(V) (1 - \lambda)}{d} + \frac{s(V) (1 + \eps)}{d} ,   \]
    where in the second step we substituted $Y := \diag(y) = \diag(2 1_{B} - 1_{n})$, and in the final step we used $\infty$-expansion \cref{d:frameInftyExpansion} to bound $\|V Y V^{T}\|_{\op}$ and the $\eps$-doubly balanced condition \cref{d:DeltaOpError} to bound $\|V V^{T}\|_{\op}$. By the same argument, we can lower bound
    \[ \frac{\alpha_{\min}}{d} = \min_{|B| = n/2} \langle x x^{T}, V V^{T} - (V V^{T} - 2 V_{B} V_{B}^{T}) \rangle \geq \frac{s(V) (1 - \eps)}{d} - \frac{s(V) (1 - \lambda)}{d} .  \]
\end{proof}

Our plan is to show pseudorandomness for Gaussian frames, and then show that this implies pseudorandomness for random unit vectors. 
For this, we require that the normalization step does not affect the pseudorandom property too much. 

\begin{lemma} \label{l:normalizationPseudoR}
    For frame $G$ that is $(\alpha_{\min}, \alpha_{\max}, \beta)$-pseudorandom, normalized frame $v_{j} = \frac{g_{j}}{\|g_{j}\|_{2}}$ has size $s(V) = n$ and is $(\alpha_{\min}(V), 2 \beta)$-pseudorandom for 
    \[ \alpha_{\min}(V) \geq s(V) \frac{\alpha_{\min}}{2 \alpha_{\max}} .    \]
\end{lemma}
\begin{proof}
    Consider $B \subseteq [n]$ with $|B| = 2 \beta n$. 
    By relabeling, we assume $B = \{1, ..., 2 \beta n\}$ and the norms are in increasing order $\|g_{1}\|_{2} \leq ... \leq \|g_{|B|}\|_{2}$. Letting $S := \{1, ..., \beta n\}$ be the half containing $\|g_{j}\|_{2}$ of smaller norm, we can lower bound 
    \begin{align} \label{eq:normPseudo1} V_{B} V_{B}^{T} = \sum_{j \in B} \frac{g_{j} g_{j}^{T}}{\|g_{j}\|_{2}^{2}} \succeq \sum_{j \in S} \frac{g_{j} g_{j}^{T}}{\|g_{j}\|_{2}^{2}} \succeq \frac{1}{\|g_{|S|}\|_{2}^{2}} G_{S} G_{S}^{T}  ,   \end{align}
    where the first step is by definition $v_{j} := \frac{g_{j}}{\|g_{j}\|_{2}}$, in the second step we consider the smaller subset $S \subseteq B$, and in the final step we use that the norms are in increasing order so $\|g_{j \in S}\|_{2} \geq \|g_{|S|}\|_{2}^{2}$. We can lower bound $G_{S} G_{S}^{T}$ using pseudorandomness, so we require an upper bound for $\|g_{|S|}\|_{2}^{2}$:
    \[ \|g_{|S|}\|_{2}^{2} \leq \E_{j \in B-S} \|g_{j}\|_{2}^{2} = \frac{1}{\beta n} \langle I_{d}, G_{B-S} G_{B-S}^{T} \rangle \leq \frac{1}{\beta n} \alpha_{\max} \beta = \frac{\alpha_{\max}}{n} ,    \]
    where in the first step we used that the norms are in increasing order so $\|g_{|S|}\|_{2} \leq \|g_{j \in B-S}\|_{2}$, and in the third step we used the pseudorandomness upper bound according to \cref{d:framePseudo} for $G$ with subset $|B-S| = \beta n$. Plugging this into the previous calculation gives
    \begin{align} \label{eq:normPseudo2} 
    V_{B} V_{B}^{T} 
    \succeq \frac{1}{\|g_{|S|}\|_{2}^{2}} G_{S} G_{S}^{T} 
    \succeq \frac{n}{\alpha_{\max}} \frac{\alpha_{\min} \beta}{d} I_{d} = n \frac{\alpha_{\min}}{2 \alpha_{\max}} \frac{|B|}{n} \frac{1}{d} I_{d}, 
    \end{align}
    where the first step was from \cref{eq:normPseudo1}, in the second step we used the upper bound $\|g_{|S|}\|_{2}^{2} \leq \frac{\alpha_{\max}}{n}$ and the lower bound $G_{S} G_{S}^{T} \succeq \frac{\alpha_{\min} \beta}{d} I_{d}$ by pseudorandomness of $G$ for $|S| = \beta n$, and in the final step we used that $|B| = 2 \beta n$. 
    Since $B \subseteq [n]$ was arbitrary, and noting $s(V) = n$ by normalization, this verifies $\alpha_{\min}$-pseudorandomness according to \cref{d:framePseudo}. 
\end{proof}

In the remainder, we prove \cref{p:GaussianPseudo} showing Gaussian frames satisfy the pseudorandom property with high probability. 

    \subsection{Preliminaries: Concentration and Approximation}

    We will use Gaussian concentration and some standard approximation arguments. 
  
    \begin{theorem} [Corollary 5.35 of \cite{vershynin2010introduction}] \label{t:GaussianConcentration}
For $d \leq m$, let $G \in \R^{d \times m}$ be a random matrix whose entries are independent standard normal random variables $G_{ij} \sim N(0,1)$ for $i \in [d], j \in [m]$. 
Then for any $\theta > 0$,
\[ \sqrt{m} - \sqrt{d} - \theta \leq \sigma_{\min}(G) \leq 
\sigma_{\max}(G) \leq \sqrt{m} + \sqrt{d} + \theta     \]
with probability at least $1 - 2 e^{-\theta^{2}/2}$.
\end{theorem}  

      $\sigma_{\min}(G) \geq 0$ always, so the lower tail bound becomes trivial for $\theta > \sqrt{m}$. 
      This implies the smallest failure probability we can achieve with a non-trivial lower bound is $\leq \exp( - m/2 )$.
      In order to get higher probability guarantees, we use the following stronger result for the lower tail:


  \begin{lemma} [Fact 4.5.7(3) in \cite{KLLR}] \label{l:chiMultLowerTail}
  For $g \sim N(0,I_{m})$ and any $c \geq 2$,  
  \[ Pr[ \|g\|_{2}^{2} \leq e^{-c} m ] \leq \exp( -cm/4 )  .  \]
  \end{lemma}

  We also require the following standard approximation arguments. 
  

    \begin{fact} [see e.g. \cite{TaoBook}] \label{f:setBound}
        For $0 \leq \beta \leq 1/2$, ${n \choose \beta n} \leq 2^{\beta n(1 + \log_{2}(1/\beta))}$.         
    \end{fact}

    \begin{fact} [Lemma 4.10 in \cite{PisierBook}] \label{f:netBound} 
    $\mathcal{N}$ is called an $\eta$-net of $S^{d-1}$ if, for any $x \in S^{d-1}$ there is $y \in \mathcal{N}$ such that $\|y - x\|_{2} \leq \eta$. 
    For any $\eta > 0$, there is an $\eta$-net $\mathcal{N} \subseteq S^{d-1}$ with cardinality
    \[ |\mathcal{N}| \leq (1 + 2/\eta)^{d} .    \]
    \end{fact}

  \begin{lemma} \label{l:netMultLowerBound}
  For $A \in \R^{d \times m}$, if $\mathcal{N}$ is an $\eta$-net of $S^{d-1}$, then 
  \[ \sigma_{\min}(A) \geq \inf_{x \in \mathcal{N}} \|x^{T} A\|_{2} - \eta \|A\|_{\op} .   \]
  \end{lemma}

  We omit the proof of the above as it is standard. 
  Note that $\sigma_{\min} \geq 0$ always, so the above bound is only non-trivial when $\eta < \inf_{x \in \mathcal{N}} \|x^{T} A\|_{2}/\|A\|_{\op}$. 

  \subsection{Proof of Gaussian Pseudorandomness} \label{ss:GaussianPseudo}

  In this subsection, we use the Gaussian concentration to prove pseudorandomness. 

\begin{theorem} \label{p:GaussianPseudoR}
    Let $G \in \R^{d \times n}$ be a Gaussian frame with entries $G_{ij} \sim N(0,1)$, and consider $0 < \beta \leq 1/2$. If $n \gtrsim d/\beta$ large enough, then $G$ is $(\alpha_{\min}, \alpha_{\max}, \beta)$-pseudorandom according to \cref{d:framePseudo} with
    \[ \alpha_{\min} \gtrsim nd \cdot \beta^{O(1)}, \quad \alpha_{\max} \lesssim nd ( 1+\log(1/\beta)) ,  \]
    with probability $\geq 1 - \exp( - \beta n )$. 
\end{theorem}

    We first prove spectral upper and lower bounds for a fixed $B \subseteq [n]$, and then show pseudorandomness by a union bound over all subsets. 

    \begin{lemma} \label{l:GaussiansubsetSpectralBounds}
        Let $G \in \R^{d \times m}$ be a random matrix standard Gaussian entries $G_{ij} \sim N(0,1)$. If $m \gtrsim d$ is large enough, for any $c \geq 4$, 
        \[  \frac{e^{-c}}{9} m \cdot I_{d} \preceq G G^{T} \preceq (2 + c) m \cdot I_{d}   \]
        with probability $\geq 1 - 2 \exp( - c m / 8)$. 
    \end{lemma}
    \begin{proof}
    Applying \cref{t:GaussianConcentration} with $\theta = \sqrt{c m}$ for some $c \geq 4$ chosen later, we have
    \begin{align} \label{eq:GSubsetUB} 
    \sigma_{\max}(G) \leq \sqrt{m} + \sqrt{d} + \theta 
    = \sqrt{m} (1 + \sqrt{d/m} + \sqrt{c} ) 
    \leq \sqrt{m ( 2 + c) } ,  
    \end{align}
    with failure probability $\leq \exp( - c m )$, where in the last step we used $m \gtrsim d$ and $c \geq 4$. 

    Next, we want to apply a net argument to lower bound $\sigma_{\min}$. For a fixed $x \in S^{d-1}$, we can apply \cref{l:chiMultLowerTail} to show 
    \begin{align} \label{eq:GSubsetLB}
    Pr[ \|x^{T} G\|_{2}^{2} \leq e^{-c} m ] \leq \exp( - c m / 4) ,
    \end{align}
    where we used that $x^{T} G$ is distributed according to $N(0,I_{m})$ by orthogonal invariance of $G$ and the fact that $x \in S^{d-1}$. 
    Now let $\mathcal{N} \subseteq S^{d-1}$ be an $\eta$-net with $\eta := \frac{2}{3} \sqrt{e^{-c}/(2+c)}$ for the same $c$ as above. 
    By \cref{f:netBound} we can bound the cardinality as
    \[ |\mathcal{N}| \leq (1 + 2/\eta)^{d} = (1 + 3 \sqrt{e^{c}(2+c)})^{d} \leq e^{c d} ,     \]
    where in the last step we use the assumption $c \geq 4$. 
    Applying the union bound, we have that the lower bound in \cref{eq:GSubsetLB} holds simultaneously for all elements of $\mathcal{N}$ with failure probability at most
    \[ |\mathcal{N}| \exp( - c m / 4) \leq \exp( cd - c m / 4) \leq \exp(- c m / 8 ) ,    \]
    where we used the cardinality bound for $|\mathcal{N}|$, and in the last step we used the assumption $m \gtrsim d$.
    
    Finally, we can lower bound $\sigma_{\min}$ by approximation:
    \[ \sigma_{\min}(G) \geq \min_{x \in \mathcal{N}} \|x^{T} G\|_{2} - \eta \|G\|_{\op} \geq \sqrt{ e^{-c} m} - \frac{2}{3} \sqrt{e^{-c}/(2+c)} \sqrt{ m (2 + c)} = \frac{1}{3} \sqrt{e^{-c} m} ,    \]
    where the first step was by \cref{l:netMultLowerBound}, and in the second step we used the lower bound over the net from \cref{eq:GSubsetLB} and the upper bound for $\|G\|_{\op}$ from \cref{eq:GSubsetUB}, as well as our choice of $\eta = \frac{2}{3} \sqrt{e^{-c} (2+c)}$. 
    The spectral bounds for $G G^{T}$ follow from these singular value bounds, and the failure probability can be bounded by a union bound for the upper and lower bounds.  
        \end{proof}

        We can now prove the pseudorandom property for Gaussian frames. 

    \begin{proof} [Proof of \cref{p:GaussianPseudoR}]
    We want to apply the above to every subset $B \subseteq [n]$ with $|B| = \beta n$. Therefore we need a union bound over ${n \choose \beta n} \leq 2^{\beta n (1 - \log_{2} \beta)}$ sets (\cref{f:setBound}).
    Fixing $B \subseteq [n]$ with $|B| = \beta n$ and applying \cref{l:GaussiansubsetSpectralBounds} with $c = 8(2 - \log \beta)$ gives
    \[ \frac{(\beta/e^{2})^{8}}{9} \frac{\beta n d}{d} \cdot I_{d} = \frac{e^{-c}}{9} |B| \cdot I_{d} \preceq G_{B} G_{B}^{T} \preceq (2+c) |B| \cdot I_{d} = (18 + 8 \log(1/\beta)) \frac{\beta n d}{d} \cdot I_{d}  .   \]
    By union bound, we get spectral bounds for all subsets $|B| = \beta n$ with total failure probability
    \[ \leq 2^{\beta n (1 - \log_{2} \beta)} \cdot 2 \exp( - c \beta n / 8 ) \leq \exp( - \beta n ) ,  \] 
    as $c = 8(2 - \log \beta)$. 
    By \cref{d:framePseudo}, this implies $(\alpha_{\min}, \alpha_{\max}, \beta)$-pseudorandomness with
    \[ \alpha_{\min} \gtrsim nd \cdot \beta^{8}, \qquad \alpha_{\max} \lesssim nd (1 + \log_{2} (1/\beta)) .   \]
        \end{proof}

\section{Proof of $\infty$-expansion implies Quantum Expansion} \label{app:InftyImpliesQuantumExpansion}

    In this section, we prove \cref{t:InftyImpliesQ} showing $\infty$-expansion implies quantum expansion.
    We will use the following simple consequence of the pseudorandom property. 

    \begin{fact} \label{f:pseudoBiggerPseudo}
    If $V \in \R^{d \times n}$ is $(\alpha_{\min}, \alpha_{\max}, \beta)$-pseudorandom according to \cref{d:framePseudo}, then for all subspaces $A \subseteq \R^{d}$ and subsets $B \subseteq [n]$ with $|B| \geq \beta n$, 
    \[ \alpha_{\min} \frac{\dim(A)}{d} \frac{|B|}{n} \leq \|P_{A} V_{B}\|_{F}^{2} = \langle P_{A}, V_{B} V_{B}^{T} \rangle \leq \alpha_{\max} \frac{\dim(A)}{d} \frac{|B|}{n}   \]
    \end{fact}
    \begin{proof}
        This follows by a simple averaging: for $|B| \geq \beta n$,  
        \[ \frac{1}{|B|} V_{B} V_{B}^{T} = \E_{T} \frac{1}{|T|} V_{T} V_{T}^{T}   \]
        where $T \subseteq B$ is a uniformly random subset of size $|T| = \beta n$. Using \cref{d:framePseudo} of pseudorandomness for these smaller set and rearranging gives:
        \[ \alpha_{\min} \frac{|B|}{nd} I_{d} = \frac{|B|}{\beta n} \frac{\alpha_{\min} \beta}{d} I_{d} \preceq V_{B} V_{B}^{T} \preceq \frac{|B|}{\beta n} \frac{\alpha_{\max} \beta}{d} I_{d} = \alpha_{\max} \frac{|B|}{n d} I_{d} .  \]
        The bounds for inner product $\langle P_{A}, V_{B} V_{B}^{T} \rangle$ follow directly from these spectral bounds. 
    \end{proof}

    We can now show that $\infty$-expansion implies quantum expansion. Our strategy will be to follow the Cheeger's style argument of \cite{FM20}. 

    \begin{theorem} \label{t:InftyImpliesQR}
        For doubly balanced frame $V \in \R^{d \times n}$, if $V$ satisfies $(1-\lambda)$-$\infty$-expansion according to \cref{d:frameInftyExpansion}, then it satisfies $(1-\Omega(\lambda^{2}))$-quantum expansion according to \cref{d:frameQuantumExpansion}. 
    \end{theorem}
    \begin{proof}
        Corollary A.4 in \cite{FM20} shows that doubly balanced $V$ satisfies $(1-\text{ch}(V)^{2})$-quantum expansion for 
        \[ \text{ch}(V) := \min_{A \subseteq \R^{d}, B \subseteq [n]} \frac{\|(I_{d} - P_{A}) V_{B}\|_{F}^{2} + \|P_{A} V_{\overline{B}}\|_{F}^{2} }{\|P_{A} V\|_{F}^{2} + \|V_{B}\|_{F}^{2}}  ,   \]
        where the minimum is over all subspaces $A \subseteq \R^{d}$ and subsets $B \subseteq [n]$ satisfying $\frac{\dim(A)}{d} + \frac{|B|}{n} \leq 1$.

        Therefore the result follows if we can show $(1-\lambda)$-$\infty$-expansion implies $\text{ch}(V) \gtrsim \lambda$. 
        Fixing subspace $A \subseteq \R^{d}$ and subset $B \subseteq [n]$, let $a := \frac{\dim(A)}{d}, b := \frac{|B|}{n}$ and assume $a+b \leq 1$. 
        The goal is to show 
        \begin{align} \label{eq:cheegerCondition} 
        \langle I_{d} - P_{A}, V_{B} V_{B}^{T} \rangle + \langle P_{A}, V_{\overline{B}} V_{\overline{B}}^{T} \rangle \gtrsim s(V) \lambda (a+b)  .  
        \end{align}

        Note that $\infty$-expansion and quantum expansion are both homogeneous conditions, so we assume $s(V) = 1$ for simplicity. 
        By \cref{l:pseudoInftyRelation}, $(1-\lambda)$-$\infty$-expansion implies that $V$ is $(\alpha_{\min} \geq \lambda, \frac{1}{2})$-pseudorandom. By \cref{f:pseudoBiggerPseudo}, this implies, for all subspaces $A \subseteq \R^{d}$ and subsets $|B| \geq n/2$, 
        \begin{align} \label{eq:InftyImpliesPseudo} 
         \langle P_{A}, V_{B} V_{B}^{T} \rangle \geq \alpha_{\min} \frac{\dim(A)}{d} \frac{|B|}{n}
         \geq \lambda \frac{\dim(A)}{d} \frac{|B|}{n} .
        \end{align}
        We can use this to lower bound \cref{eq:cheegerCondition} by some case analysis:
        
        \begin{itemize}
            \item If $b \geq 1/2$: we use the pseudorandom condition for subset $B$ to lower bound
            \[ \| (I_{d} - P_{A}) V_{B}\|_{F}^{2} \geq \alpha_{\min} (1-a) b \geq \lambda \frac{(1-a) b}{a+b} (a+b) \geq \frac{\lambda}{4} (a+b)  ,    \]
            where the first step is by pseudorandomness \cref{eq:InftyImpliesPseudo} for $b \geq 1/2$, and in the final step we use the assumptions $a+b \leq 1, b \geq 1/2$ so $1-a \geq 1/2$ and $b \geq a$. 
            
            \item If $b \leq 1/2$ and $b \leq 2a$: we use the pseudorandom property for complement subset $\overline{B}$
            \[ \|P_{A} V_{\overline{B}} \|_{F}^{2} \geq \alpha_{\min} a (1-b) \geq \lambda \frac{a (1-b)}{a+b} (a+b) \geq \frac{\lambda}{6} (a+b) ,      \]
            where in the last step we used our case assumptions to bound $1-b \geq 1/2$ and $6 a \geq a+b$. 

            \item Finally, if $b \leq 1/2$ and $b \geq 2a$: we use the doubly balanced property to lower bound
            \[ \|(I_{d} - P_{A}) V_{B} \|_{F}^{2} \geq \|(I_{d} - P_{A}) V\|_{F}^{2} - \|V_{\overline{B}}\|_{F}^{2} \geq (1-a) - (1-b) = \frac{b-a}{a+b} (a+b) \geq \frac{a+b}{3}  ,     \]
            where in the first step we used $I_{d} - P_{A}$ is a projection so $\|(I_{d} - P_{A}) V_{\overline{B}}\|_{F} \leq \|V_{\overline{B}}\|_{F}$, the second step was by the doubly balanced condition, and in the final step we used $b \geq 2a$ so $\frac{b-a}{a+b} \geq \frac{1}{3}$.
        \end{itemize}
        Combining the above three cases gives \cref{eq:cheegerCondition} for arbitrary $A,B$ with $a+b \leq 1$. This implies $\ch(V) \geq \frac{\lambda}{6}$ according to the definition, which proves quantum expansion via Corollary A.4 of \cite{FM20} as stated above. 
    \end{proof}
\end{document}